\documentclass[12pt,reqno]{article}

\usepackage[usenames]{color}
\usepackage{amssymb}
\usepackage{amscd}

\usepackage[colorlinks=true,
linkcolor=webgreen,
filecolor=webbrown,
citecolor=webgreen]{hyperref}

\definecolor{webgreen}{rgb}{0,.5,0}
\definecolor{webbrown}{rgb}{.6,0,0}

\usepackage{color}
\usepackage{fullpage}
\usepackage{float}

\usepackage{graphics,amsmath,amssymb}
\numberwithin{equation}{section}
\usepackage{amsthm}
\usepackage{amsfonts}
\usepackage{latexsym}

\DeclareMathOperator{\Li}{Li}

\begin{document}

%\begin{center}
%\epsfxsize=4in
%\leavevmode\epsffile{logo129.eps}
%\end{center}

\theoremstyle{plain}
\newtheorem{theorem}{Theorem}
\newtheorem{corollary}[theorem]{Corollary}
\newtheorem{lemma}[theorem]{Lemma}
\theoremstyle{remark}
\newtheorem{remark}[theorem]{Remark}

\begin{center}
\vskip 1cm{\LARGE\bf The integrals in Gradshteyn and Ryzhik. Part 1: An Addendum \\
\vskip .11in }

\vskip 0.8cm

{\large
Kunle Adegoke \\
Department of Physics and Engineering Physics, \\ Obafemi Awolowo University, Ile-Ife, Nigeria \\
\href{mailto:adegoke00@gmail.com}{\tt adegoke00@gmail.com}

\vskip 0.15 in

Robert Frontczak\footnote{Statements and conclusions made in this paper by R.~Frontczak are entirely those of the author.
They do not necessarily reflect the views of LBBW.} \\ Landesbank Baden-W\"urttemberg, Stuttgart, Germany \\
\href{mailto:robert.frontczak@lbbw.de}{\tt robert.frontczak@lbbw.de}

\vskip 0.15 in

Taras Goy \\
Faculty of Mathematics and Computer Science \\ Vasyl Stefanyk Precarpathian National University, Ivano-Frankivsk, Ukra\-ine \\
\href{mailto:taras.goy@pnu.edu.ua}{\tt taras.goy@pnu.edu.ua}}
\end{center}

\vskip .15in

\begin{abstract}
We present another generalization of a logarithmic integral studied by V.~H. Moll in 2007.
The family of integrals contains three free parameter and its evaluation involves the harmonic numbers.
\end{abstract}

\medskip


\medskip


\section{Motivation}

The classical ``Table of Integrals, Series and Products'' by Gradshteyn and Ryzhik \cite{GrRy07}
contains a huge range of values of definite integrals. In a series of papers beginning in 2007,
Moll, Amdeberhan, Medina, Boyadzhiev, Vignat and others established, corrected and generalized many of these formulas.
Part 30 \cite{Part30} is probably one of the most recent papers in this series,
although Boros and Moll \cite{BoMo} formulated the desire to prove all the formulas from \cite{GrRy07},
which is a hard and tortuous task. Moll \cite{Moll-Mono1,Moll-Mono2} has written excellent books dealing with special integrals of Gradshteyn and Ryzhik \cite{GrRy07}.

Formula 4.232.3 in \cite{GrRy07} states that
\begin{equation}\label{log_int}
\int_0^\infty \frac{\ln{x}\,dx}{(x+a)(x-1)} = \frac{\pi^2 + \ln^2{a}}{2(a+1)}, \quad a>0.
\end{equation}
This formula is interesting as it allows to derive some related values as well.
For instance, with $a=\alpha^2$ and $a=\alpha^{-2}$ ($\alpha=\frac{1+\sqrt{5}}{2}$) upon combining
we get, respectively, the formulas
\begin{equation*}
\int_0^\infty \frac{\ln{x}\,dx}{(x-1)(x^2+3x+1)} = \frac{1}{5} \int_0^\infty \frac{(2x+3)\ln x\,dx}{(x-1)(x^2+3x+1)}
= \frac{\pi^2 + 4 \ln^2 \alpha}{10},
\end{equation*}
\begin{equation*}
\int_0^\infty  \frac{(x+1)\ln{x}\,dx}{(x-1)(x^2+3x+1)}  = \frac{\pi^2 + 4 \ln^2\alpha}{5}.
\end{equation*}
Also, with $a=\alpha$ and $a=-\beta=\alpha^{-1}$, in turn, from \eqref{log_int} we have the following interesting integrals:
\begin{align*}
\int_{0}^{\infty} \frac{(x+1-\alpha)\ln{x\,dx}}{(x^2+x-1)(x-1)}&=\frac{\pi^2+\ln^2\alpha}{2\alpha^2},\\
\int_{0}^{\infty} \frac{(x-\alpha)\ln{x}\,dx}{(x^2-x-1)(x-1)}&=\frac{\pi^2+\ln^2\alpha}{2\alpha}.
\end{align*}

In the very first paper of the above series \cite{Moll}, Moll generalized \eqref{log_int} by considering the family of logarithmic integrals
\begin{equation*}
f_n (a) = \int_0^\infty \frac{\ln^{n-1}{x}\,dx}{(x+a)(x-1)}, \quad n\geq 2,\,\, a>0.
\end{equation*}
Moll proved that
\begin{align*}
f_n(a) & =  \frac{(-1)^n (n-1)!}{a+1}\Big ( (1-(-1)^{n-1})\zeta(n) -\Li_n\Big (-\frac{1}{a}\Big ) + (-1)^{n-1} \Li_n(-a)\Big ) \nonumber \\
& =  \frac{(-1)^n (n-1)!}{a+1}\Big ( (1-(-1)^{n-1})\zeta(n) \\
& \quad-  \frac{1}{n(a+1)} \sum_{j=0}^{\lfloor n/2 \rfloor} (-1)^{j}\binom {n}{2j}(2^{2j}-2)  \pi^{2j} B_{2j} \ln^{n-2j}{a} \Big ),
\end{align*}
where $\zeta(s)$ is the Riemann zeta function, $\Li_n(z)$ is the polylogarithm and $B_n$ are the Bernoulli numbers.

In this paper we provide an addendum to Moll's paper by considering the different family of integrals
\begin{equation}\label{int1}
F (m,k,a) = \int_0^\infty \frac{x^m\ln{x}\,dx}{(x-1)(x+a)^{k+m+1}},
\end{equation}
where the three parameter satisfy $m,k\in\mathbb{N}_0$ and $a>0$.
\newpage
We require the following lemma in the sequel.
\begin{lemma} If $c$ is an arbitrary constant and $s\in \mathbb{R}$, then
	\begin{equation}\label{eq.al7j6j4}
	\frac{d^k}{da^k}\left(\frac{{a + c}}{{(ax + 1)^s }}\right) = \frac{{( - 1)^k k!x^{k - 1} }}{{(ax + 1)^{k + s} }}\left( {(a + c)x\binom{s + k - 1}{s-1} - (ax + 1)\binom{s + k - 2}{s - 1}} \right).
	\end{equation}
\end{lemma}
\begin{proof} Leibnitz rule gives
	\[
	\frac{{d^k }}{{da^k }}\left(\frac{{a + c}}{{(ax + 1)^s }}\right) = (a + c)\frac{{d^k }}{{da^k }}\left(\frac{1}{{(ax + 1)^s }}\right) + k {\frac{d}{{da}}(a + c)}\frac{{d^{k - 1} }}{{da^{k - 1} }}\left(\frac{1}{{(ax + 1)^s }}\right),
	\]
	from which~\eqref{eq.al7j6j4} follows, since
	\[
	\frac{d^k}{da^k}\left(\frac{1}{{(ax + 1)^s }}\right) = ( - 1)^k \frac{{(s + k - 1)!x^k }}{{(s - 1)!(ax + 1)^{k + s} }}.
	\]
\end{proof}

\section{The evaluations of $F(m,k,a)$ for $m=0,1,2$}

Before deriving the general expression for $F(m,k,a)$ we study in detail some special cases.
First we prove the following formula for $F(0,k,a)$.
\begin{theorem}\label{thm_main1}
	For $k\in\mathbb{N}_0$ and $a>0$, we have
	\begin{equation}\label{Main1}
	F(0,k,a) = \frac{1}{(a+1)^{k+1}}\left(\frac{\pi^2 + \ln^2{a}}{2} + \sum_{j=0}^{k-1} \frac{\big(1+1/a\big)^{j+1}}{j+1} (H_{j} - \ln{a})\right)
	\end{equation}
	with $H_n=\sum_{k=1}^n\frac{1}{k}$, $H_0=0$, being the harmonic numbers.
\end{theorem}
\begin{proof} Starting with \eqref{log_int} we differentiate both sides $k$ times with respect to $a$ to get
	\begin{align*}
	(-1)^k k! \int_0^\infty \frac{\ln{x}\,dx}{(x+a)^{k+1}(x-1)}& = \frac{1}{2} \frac{d^k}{d a^k} \left(\frac{\pi^2 + \ln^2{a}}{a+1}\right)\\
	&= \frac{1}{2} \sum_{j=0}^k \binom {k}{j} ((a+1)^{-1})^{(j)}(\pi^2+\ln^2{a})^{(k-j)},
	\end{align*}
	where we have used the Leibniz rule for derivatives. We have
	\begin{equation}\label{help_der1}
	\frac{d^k}{d a^k}\left(\frac{1}{x+a}\right)  = \frac{(-1)^k \, k!}{(x+a)^{k+1}}, \quad k\geq 0.
	\end{equation}
	Now, assuming that
	\begin{equation*}
	\frac{d^k}{d a^k} ( \pi^2 + \ln^2{a} ) = \frac{X_k}{a^{k}} + \frac{Y_k\,\ln{a}}{a^k}
	\end{equation*}
	we get the recurrences, for $k\geq 1$,
	\begin{equation*}
	X_{k+1} = Y_k - k X_k \quad\mbox{and}\quad Y_{k+1} = - k Y_k,
	\end{equation*}
	with $X_1=0$ and $Y_1=2$. The recurrence for $Y_k$ is solved straightforwardly and the result is $Y_k=2 (-1)^{k-1}\,(k-1)!$.
	This gives
	\begin{align*}
	X_{k+1} & = \sum_{j=0}^{k-1} (-1)^j j! \binom {k}{j} Y_{k-j} \\
	& = 2 (-1)^{k-1} k! \sum_{j=0}^{k-1}\frac{1}{k-j} = 2 (-1)^{k-1} k!\, H_k,
	\end{align*}
	and finally, for $k\geq 1$
	\begin{equation}\label{eq.bbwvmo2}
	\frac{d^k}{d a^k} ( \pi^2 + \ln^2{a}) = \frac{2 (-1)^k (k-1)!}{a^{k}} (H_{k-1}-\ln{a}).
	\end{equation}
	The formula \eqref{Main1} follows upon simplifications.
\end{proof}

For $k=0$ in \eqref{Main1} we get \eqref{log_int}. The next two cases are
\begin{equation*}
F(0,1,a) = \frac{a(\pi^2 + \ln^2{a}) - 2(a+1) \ln{a}}{2 a^2 (a+1)^2},
\end{equation*}
and
\begin{equation*}
F(0,2,a) = \frac{a^2(\pi^2 + \ln^2{a}) - (a+1)(3a+1) \ln{a} + (a+1)^2}{2 a^2 (a+1)^3}.
\end{equation*}
\begin{corollary}\label{Cor2.2}
	For $k\in\mathbb{N}_0$, we have
	\begin{equation}
	\begin{split}\label{CorMain1}
	& \int_0^\infty\frac{\ln{x}\sum\limits_{j=0}^{k+1} \binom {k+1}{j}  L_{2(k+1-j)}x^j}{(x-1)(x^2+3x+1)^{k+1}}\,dx
	= \frac{\pi^2 + 4\ln^2{\alpha}}{2\cdot 5^{k/2}} 
	\begin{cases}
	\frac{L_{k+1}}{\sqrt5}, & \text{\rm{if $k$ is odd;}} \\
	F_{k+1}, & \text{\rm{if $k$ is even}}
	\end{cases}	\\
	& \qquad \qquad +\frac{(-1)^k}{5^{k/2}} \sum_{j=0}^{k-1} \frac{(-1)^{j}5^{j/2}}{j+1} \Big ( H_{j}\big(L_{k+2+j}+\alpha^{k+2+j}((-1)^{k-j}-1)\big)\\
	& \qquad \qquad- 2\ln{\alpha} \big(L_{k+2+j}-\alpha^{k+2+j}\big((-1)^{k-j}+1)\big)\Big )
	\end{split}
	\end{equation}
	and
	\begin{equation}
	\begin{split} \label{CorMain2}
	& \int_0^\infty \frac{\ln x\sum\limits_{j=0}^{k+1} \binom {k+1}{j} F_{2(k+1-j)}x^j}{(x-1)(x^2+3x+1)^{k+1}}\, dx
	= \frac{\pi^2 + 4\ln^2{\alpha} }{2\cdot 5^{(k+1)/2}} 
	\begin{cases}
	{F_{k+1}}, & \text{\rm{if $k$ is odd;}} \\
	\frac{L_{k+1}}{\sqrt5}, & \text{\rm{if $k$ is even}}
	\end{cases}\\
	& \qquad \qquad -\frac{(-1)^k}{5^{(k+1)/2}} \sum_{j=0}^{k-1} \frac{(-1)^{j}5^{j/2}}{j+1} \Big ( H_{j}\big(L_{k+2+j}-\alpha^{k+2+j}((-1)^{k-j}+1)\big) \\
	& \qquad \qquad  - 2\ln{\alpha} \big(L_{k+2+j}+\alpha^{k+2+j}((-1)^{k-j}-1)\big)\Big )
	\end{split}
	\end{equation}
	with $F_n \,(L_n)$ being the Fibonacci (Lucas) numbers and where $\alpha=(1+\sqrt{5})/2$ is the golden ratio.
\end{corollary}
\begin{proof} To get \eqref{CorMain1} insert $a=\alpha^2$ and $a=\beta^2=\alpha^{-2}$ in \eqref{Main1}, respectively, and add the expressions.
	When simplifying use the relations $\alpha^2+1=\sqrt{5}\alpha$ and $\beta^2+1=-\sqrt{5}\beta$ as well as
	\begin{equation*}
	\alpha^{k+1}+(-1)^{k+1}\beta^{k+1} = \begin{cases}
	L_{k+1}, & \text{if $k$ is odd;} \\
	\sqrt5F_{k+1}, & \text{if $k$ is even.}
	\end{cases}
	\end{equation*}
	Identity \eqref{CorMain2} is obtained by subtraction using
	\begin{equation*}
	\alpha^{k+1}-(-1)^{k+1}\beta^{k+1} =
	\begin{cases}
	L_{k+1}, & \text{if $k$ is even;} \\
	\sqrt5 F_{k+1}, & \text{if $k$ is odd.}
	\end{cases}
	\end{equation*}
\end{proof} When $k=0$ and $k=1$ then Corollary \ref{Cor2.2} yields the following results as particular cases:
\begin{align}
\int_0^\infty \frac{(2x+3)\ln x\,dx }{(x-1)(x^2+3x+1)}& = \frac{\pi^2}{2} + 2\ln^2{\alpha},\label{My1}\\
\int_0^\infty \frac{\ln{x}\,dx }{(x-1)(x^2+3x+1)}&= \frac{\pi^2}{10} + \frac{2}{5} \ln^2{\alpha},\label{My2}\\
\int_0^\infty \frac{(2x^2+6x+7)\ln x\,dx }{(x-1)(x^2+3x+1)^{2}}&
= \frac{3\pi^2}{10} + \frac{6}{5}\ln^2\alpha + \frac{8}{\sqrt{5}}\ln\alpha\label{My3}
\end{align}
and
\begin{equation}
\int_0^\infty \frac{(2x+3)\ln{x}\,dx}{(x-1)(x^2+3x+1)^2}
= \frac{\pi^2}{10} + \frac{2}{5} \ln^2{\alpha}+ \frac{4}{\sqrt{5}}\ln{\alpha}.\label{My4}
\end{equation}

Since, as is easily shown from \eqref{My1} and \eqref{My2}, 	
$$
\int_{0}^{\infty} \frac{x\ln{x}\,dx}{(x-1)(x^2+3x+1)}=\frac{\pi^2+4\ln^2\alpha}{10},
$$
it follows that  
$$
\int_{0}^{\infty} \frac{(sx+q)\ln{x}\,dx}{(x-1)(x^2+3x+1)} = (s+q)\frac{\pi^2+4\ln^2\alpha}{10},
$$
for arbitrary $s$ and $q$. Similarly, from \eqref{My3} and \eqref{My4} we have
$$
\int_{0}^{\infty} \frac{(sx^2+qx+r)\ln{x}\,dx}{(x^2+3x+1)^2}=\frac{2(s-r)}{\sqrt5}\ln{\alpha},
$$ 
for arbitrary $s$, $q$ and $r$.
\begin{corollary}
	If $k\in \mathbb{N}_0$ and $r$ is an even integer, then
	\begin{equation}\label{eq.cqkdxx9}
	\begin{split}
	\int_0^\infty \frac{\ln x\sum\limits_{j = 0}^{k + 1} \binom{k + 1}j L_{2rj}x^{k + 1 - j}}{(x - 1)(x^2  + L_{2r}x + 1)^{k + 1}}&\,dx 
	=\frac{L_{r(k + 1)}}{2L_r^{k + 1}}
	(\pi ^2  + 4r^2 \ln ^2 \alpha )	\\
	&+ \frac{1}{L_r^k}\sum_{j = 0}^{k - 1} \frac{L^j_r}{j+1}
	\left(L_{r(k+2+ j)} H_{j} + 2\sqrt 5r  F_{r(k + 2+ j)}\ln \alpha\right), 
	\end{split}
	\end{equation}
	\begin{equation}\label{eq.dk2mwuc}
	\begin{split}
	\int_0^\infty  \frac{{\ln x\sum\limits_{j = 0}^{k + 1} {\binom{k + 1}j F_{2rj}x^{k + 1 - j} } }}{{(x - 1)(x^2  + L_{2r} x + 1)^{k + 1} }}&\,dx = \frac{{F_{r(k + 1)} }}{{2L_r^{k + 1} }}(\pi ^2  + 4r^2 \ln ^2 \alpha )\\
	& + \frac{1}{L^k_r}\sum_{j = 0}^{k - 1} \frac{L^j_r}{j+1}\left(F_{r(k+ j + 2)} H_{j} + \frac{2rL_{r(k+2+j)}  }{\sqrt 5}\ln \alpha\right). 
	\end{split}
	\end{equation}
\end{corollary}
\begin{proof}
	Consider $F(0,k,\alpha ^{2r} ) \pm F(0,k,\beta ^{2r} )$, using~\eqref{int1} and~\eqref{Main1}; and the fact that if $r$ is an even integer, then
	$\alpha^{2r} + 1 =\alpha^r L_r$ and $\beta^{2r} + 1 =\beta^r L_r$.
\end{proof}
\begin{theorem}
	If $a>0$ and $k\in \mathbb{N}$, then
	\begin{equation}\label{eq.gavc1hl}
	\int_0^\infty  {\frac{\ln x\,dx}{(x + a)^{k + 1}}}  = \frac{\ln a - H_{k - 1} }{{ka^k }}.
	\end{equation}
\end{theorem}
\begin{proof} Write~\eqref{log_int} as
	\[
	2\int_0^\infty\frac{{(a + 1)\ln x\,dx}}{{(ax + 1)(x - 1)}} = \pi ^2  + \ln ^2 a;
	\]
	differentiate both sides $k$ times with respect to $a$, making use of~\eqref{eq.al7j6j4} and~\eqref{eq.bbwvmo2}. Write $1/a$ for~$a$.
\end{proof}
Note that~\eqref{eq.gavc1hl} is equivalent to Gradshteyn and Ryzhik~\cite[4.253.6]{GrRy07}; in which case the harmonic number is expressed in terms of the digamma function, thereby removing the restriction on $k$.
\begin{corollary}
	If $a,k>0$ and $m>1$, then
	\begin{equation*}\label{eq.ep8qs3l}
	\int_0^\infty  {\frac{{x^{k - 1} \ln x}}{{(x + a)^{k + m} }}\left( {x\binom{m + k -2}{m - 2} - a\binom{m + k - 2}{m - 1}} \right)dx}  = \frac{1}{{(m - 1)ka^{m - 1} }}.
	\end{equation*}
\end{corollary}
\begin{proof}
	Write $1/a$ for $a$ and $m - 1$ for $k$ in~\eqref{eq.gavc1hl}  to obtain
	\[
	\int_0^\infty \frac{{a\ln x \,dx}}{{(ax + 1)^m }} =  \frac{\ln a + H_{m - 2}}{1-m}.
	\]
	Differentiate the above expression $k$ times with respect to $a$, using~\eqref{eq.al7j6j4} and~\eqref{eq.slgmwoa}. Finally, write $a$ for $1/a$.
\end{proof}

The integral $F(1,k,a)$ is evaluated in the next theorem.
\begin{theorem}%\label{thm_main2}
	For $k\in\mathbb{N}_0$ and $a>0$, we have
	\begin{equation}
	\begin{split}\label{Main2}
	F (1,k,a) & = \frac{\pi^2 + \ln^2 a}{2(a+1)^{k+2}} + \frac{\ln a}{(k+1)(a+1)^{k+1}} \\
	& \quad + \frac{1}{(k+1)(a+1)^{k+1}} \sum_{j=0}^{k-1}  \frac{\big(1+1/a\big)^{j+1}}{j+1}\left( \frac{k-j}{a+1} \big(H_{j} - \ln a\big) - 1 \right).
	\end{split} 
	\end{equation}
\end{theorem} \begin{proof}
	We start with the observation that
	\begin{equation}\label{int_sc}
	\int_0^\infty \frac{x\ln x\,dx}{(x+a)^{2}(x-1)} = \frac{\pi^2 + \ln^2 a}{2(a+1)^2} + \frac{\ln a}{a+1}.
	\end{equation}
	This is true since we have
	\begin{equation*}
	\int \frac{x\ln x\,dx}{(x+a)^{2}(x-1)} = g(x,a)
	\end{equation*}
	with (the constant $C$ is not displayed)
	\begin{equation}
	\label{eq.slgmwoa}
	\begin{split}
	g(x,a) &= - \frac{1}{(a+1)^2}\left (\Li_2\Big (-\frac{x}{a}\Big ) + \Li_2(1-x)\right.\\
	&\left.\quad+ \ln x\ln\Big (1+\frac{x}{a}\Big ) + \frac{a(a+1)\ln x}{x+a}+(a+1)\ln\Big (1+\frac{a}{x}\Big ) \right).
	\end{split}
	\end{equation}
	Now, taking the limits $\lim_{x\rightarrow\infty} g(x,a)$ and $\lim_{x\rightarrow 0} g(x,a)$ leads us to \eqref{int_sc}.
	The remainder of the proof is the same as in Theorem \ref{thm_main1} using \eqref{help_der1} and
	\begin{equation*}%\label{help_der2}
	\frac{d^k}{d a^k} \ln a = \frac{(-1)^{k-1} \, (k-1)!}{a^{k}}, \quad k\geq 1.
	\end{equation*}
\end{proof}

For $k=0$ in \eqref{Main2} we get \eqref{int_sc}. The next two cases are
\begin{equation*}
F(1,1,a) = \frac{a(\pi^2 + \ln^2 a) + (a^2-1) \ln a - (a+1)^2}{2 a (a+1)^3}
\end{equation*}
and
\begin{equation*}
F(1,2,a) = \frac{3a^2(\pi^2 + \ln^2 a) + (a+1)(2a^2-5a-1) \ln a - 3a(a+1)^2}{6 a^2 (a+1)^4}.
\end{equation*}

To derive a formula for $F(2,k,a)$ we need the next lemma.
\begin{lemma}
	For $k\in\mathbb{N}_0$, the following formula holds
	\begin{equation*}\label{eq.s59wgad}
	\frac{d^k}{d a^k}\left( \frac{a+3}{(a+1)^2}\right) = (-1)^k \, k!\,\frac{a+3+2k}{(a+1)^{k+2}}\,.
	\end{equation*}
\end{lemma}
\begin{proof} Use $c=3$, $s=2$ and $x=1$ in~\eqref{eq.al7j6j4}.
\end{proof}
The integral $F(2,k,a)$ admits the following evaluation.
\begin{theorem}%\label{thm_main3}
	For $k\in\mathbb{N}_0$ and $a>0$, we have
	\begin{align*}
	F(2,k,a) &= \frac{\pi^2 + \ln^2 a}{2(a+1)^{k+3}} +
	\frac{1}{(k+1)(k+2)(a+1)^{k+1}} \left( \frac{a+3+2k}{a+1} \ln a + 1 \right.\\
	& \quad \left.+ \frac{1}{a}\sum_{j=0}^{k-1} \frac{\big(1+1/a\big)^j}{j+1} \left( \frac{(k-j)(k+1-j)}{a+1}
	\big(H_{j} - \ln a\big) - a-1-2(k-j)\right)\right).	
	\end{align*}
\end{theorem}
\begin{proof} The proof is similar to the previous two proofs.
\end{proof}
When $k=0$ then we get
\begin{equation*}
F(2,0,a) = \frac{\pi^2 + \ln^2 a}{2(a+1)^{3}} + \frac{a+3}{2(a+1)^2} \ln a + \frac{1}{2(a+1)}.
\end{equation*}

\section{The general case}

Here we state a general formula for $F(m,k,a)$. The structure of such a formula is indicated in the above analysis.
Our main argument is not to try to derive an explicit expression for the indefinite integral
\begin{equation*}
\int \frac{x^m\ln x\,dx}{(x-1)(x+a)^{m+1}}
\end{equation*}
but instead using the results from the first part of the paper.
\begin{theorem}%\label{thm_main4}
	For $m,k\in\mathbb{N}_0$ and $a>0$, we have
	\begin{equation}\label{Main4}
	\begin{split}
	F (m,k,a) &= \int_0^\infty \frac{{x^m\ln  x\,dx}}{(x-1)(x + a)^{k + m + 1}}\\
	& =  \frac{{\pi^2}}{{2(a + 1)^{k + m + 1}}} + \frac{{(- 1)^m}}{{2m! \binom {k + m} {m} a^{k + m + 1} }}
	\left. {\frac{{d^m }}{{db^m }}\left( {\frac{{\ln^2 b}}{{(b + 1)^{k + 1} }}} \right)} \right|_{b = 1/a} \\
	&\quad + \sum_{j = 0}^{k - 1} {\frac{ \binom {j + m} {m}}{\binom{k + m}{m}} \frac{{H_{k - j - 1} }}{{k - j}}
		\frac{{a^{j - k} }}{{(a + 1)^{j + m + 1} }}}\\
	&\quad+ \frac{{(- 1)^m }}{{m!\binom {k + m}{m} a^{k + m + 1} }}\sum_{j = 0}^{k - 1} \frac{1}{{k - j}}\left. {\frac{{d^m }}{{db^m }} \left( {\frac{{\ln b}}{{(b + 1)^{j + 1} }}} \right)} \right|_{b = 1/a}.
	\end{split}
	\end{equation}
\end{theorem}
\begin{proof}
	Using $F(0,k,1/a)$ from Theorem \ref{thm_main1} we find
	\begin{equation}\label{main4help}
	\int_0^\infty \frac{\ln x \, dx}{(x-1)(ax + 1)^{k + 1}} = \frac{\pi^2 + \ln^2 a}{2(a + 1)^{k + 1}}
	+ \sum_{j = 0}^{k - 1} \frac{H_{k - j - 1} + \ln a}{(k - j)(a + 1)^{j + 1}}.
	\end{equation}
	Differentiating \eqref{main4help} $m$ times with respect to $a$ and replacing $a$ with $1/a$ gives \eqref{Main4}.
\end{proof}

In particular, $F(m,0,a)$ equals
\begin{equation}\label{main4_k0}
\begin{split}
\int_0^\infty \frac{x^m\ln x\,dx}{(x-1)(x + a)^{m + 1}} & = \frac{\pi^2}{2(a + 1)^{m + 1}} + \frac{{(- 1)^m}}{{2m!a^{m + 1}}}
\left. {\frac{{d^m }}{{db^m }}\left( {\frac{{\ln^2 b}}{{(b + 1)}}} \right)} \right|_{b = 1/a} \nonumber \\
& = \frac{1}{2(a+1)^{m+1}}\left(\pi^2+\ln^2 a + 2\sum_{j = 0}^{m - 1} \frac{(a+1)^{j+1}}{j+1}\big(H_{j} + \ln a\big)\right),
\end{split}
\end{equation}
since
\begin{equation*}
\left.{\frac{{d^m }}{{db^m }} \left( {\frac{{\ln^2 b}}{{b + 1}}} \right)} \right|_{b = 1/a}
= (- 1)^m m! \frac{a^{m + 1} \ln^2 a}{(a + 1)^{m + 1} } +
(- 1)^m2 m! a^{m + 1} \sum_{j = 0}^{m - 1} \frac{H_{m - j - 1} + \ln a}{(m - j)(a + 1)^{j + 1} }.
\end{equation*}
\begin{theorem}
	If $m\in \mathbb{N}_0$ and $r$ is an even integer, then
	\begin{equation}\label{eq.xjvr347}
	\begin{split}
	\int_0^\infty & \frac{{\ln x\sum\limits_{j = 0}^{m + 1} {\binom{m + 1}j L_{2rj}\,x^{2m + 1 - j} }}}{{(x - 1)(x^2  + L_{2r} x + 1)^{m + 1} }}\,dx\\
	& = \frac{{L_{r(m + 1)} }}{{2L_r^{m + 1} }}(\pi ^2  + 4r^2 \ln ^2 \alpha ) + \frac{1}{L_r^m}\sum_{j = 0}^{m - 1}\frac{L_{r}^j}{j+1} \left(L_{r(m-j)} H_{j} - 2\sqrt 5 r \ln \alpha F_{r(m-j)} \right),
	\end{split}
	\end{equation}
	\begin{equation}\label{eq.i9gr6oo}
	\begin{split}
	\int_0^\infty & \frac{\ln x\sum\limits_{j = 0}^{m + 1} \binom{m + 1}j F_{2rj}\,x^{2m + 1 - j}}{{(x - 1)(x^2  + L_{2r} x + 1)^{m + 1} }}\,dx\\
	& = \frac{{F_{r(m + 1)} }}{2L_r^{m + 1}}(\pi ^2  + 4r^2 \ln ^2 \alpha ) + \frac{1}{L_r^m} 
	\sum_{j = 0}^{m - 1} \frac{L^j_r}{j+1} \left( F_{r(m-j)} H_{j}- \frac{2\sqrt 5}{5} r\ln \alpha L_{r(m-j)}\right) .
	\end{split}
	\end{equation}
\end{theorem}
\begin{proof} 	Evaluate $F(m,0,\alpha^{2r})\pm F(m,0,\beta^{2r})$.
\end{proof}
\begin{corollary}
	If $m,k\in \mathbb{N}_0$ and $r$ is an even integer, then
	\begin{equation}\label{eq.br2tvq8}
	\begin{split}
	\int_0^\infty & \frac{{x^{k+1}(x^k-1)\ln x\sum\limits_{j = 0}^{k + 1} \binom{k+1}{j} \frac{L_{2rj}}{ x^j} }}{{(x - 1)(x^2  + L_{2r} x + 1)^{k + 1} }}\,dx \\
	&\qquad\qquad\qquad \qquad   =-\frac{5}{L_r^k} \sum_{j = 0}^{k - 1} \frac{L_r^j F_{r(k+1)}}{j+1}\Bigl(F_{r(j+1)}H_{j} +\frac{2\sqrt5 r} {5} \ln \alpha L_{r(j +1)} \Bigr),
	\end{split}
	\end{equation}
	\begin{equation}\label{eq.hqyu6af}
	\begin{split}
	\int_0^\infty & \frac{x^{k+1}(x^k+1)\ln x\sum\limits_{j = 0}^{k + 1} \binom{k+1}j \frac{L_{2rj}}{x^j}}{(x - 1)(x^2  + L_{2r} x + 1)^{k + 1}}\,dx = \frac{L_{r(k + 1)}}{L_r^{k + 1}}(\pi ^2  + 4r^2 \ln ^2 \alpha )\\
	&\qquad\qquad\qquad \qquad\qquad  + \frac{L_{r(k + 1)}}{L_r^k} \sum_{j = 0}^{k - 1} \frac{L^j_{r}}{j+1}\left(L_{r(j+1)}H_{j} + 2\sqrt5 r \ln\alpha F_{r(j +1)}\right), 
	\end{split}
	\end{equation}
	\begin{equation}\label{eq.a5vog8v}
	\begin{split}
	\int_0^\infty & \frac{x^{k+1}(x^k-1)\ln x\sum\limits_{j = 0}^{k + 1} \binom{k+1}j \frac{F_{2rj}}{x^j}}{(x-1)(x^2  + L_{2r} x + 1)^{k+1}}\,dx\\
	&\qquad\qquad\qquad\qquad  = - \frac{L_{r(k + 1)}}{L_r^k} \sum_{j = 0}^{k - 1} \frac{L_r^j}{j+1}\Big( F_{r(j +1)} H_{ j}+\frac{2\sqrt 5 r}{5} \ln \alpha L_{r(j +1)}\Big),
	\end{split}
	\end{equation}
	\begin{equation}\label{eq.hcoyj7m}
	\begin{split}
	\int_0^\infty & \frac{x^{k+1}(x^k+1)\ln x\sum\limits_{j = 0}^{k + 1} \binom{k+1}{j} \frac{F_{2rj}}{x^j}}{(x - 1)(x^2  + L_{2r} x + 1)^{k + 1} }\,dx = \frac{F_{r(k + 1)} }{L_r^{k + 1}}(\pi ^2  + 4r^2 \ln ^2 \alpha ) \\
	&\qquad\qquad \qquad \qquad  + \frac{F_{r(k + 1)}}{L^k_r} \sum_{j = 0}^{k - 1} \frac{L_{r}^j}{j+1}
	\left(L_{r(1+j)} H_{j} + 2\sqrt 5 r\ln \alpha F_{r(1+j)}\right).
	\end{split}
	\end{equation}
\end{corollary}
\begin{proof}
	Set $m=k$ in~\eqref{eq.xjvr347}; subtract/add~\eqref{eq.cqkdxx9} to obtain~\eqref{eq.br2tvq8}/\eqref{eq.hqyu6af}. Similarly,~\eqref{eq.a5vog8v} and~\eqref{eq.hcoyj7m} follow from~\eqref{eq.dk2mwuc} and~\eqref{eq.i9gr6oo}. Note the use of the following identities that are valid for all integers $u$ and $v$ having the same parity:
	\begin{align*}
	F_u  + F_v  &= \begin{cases}
	L_{(u - v)/2} F_{(u + v)/2}, & \text{if $(u - v)/2$ is even;} \\
	F_{(u - v)/2} L_{(u + v)/2}, & \text{if $(u - v)/2$ is odd,}
	\end{cases}\\
	F_u  - F_v  &= \begin{cases}
	L_{(u - v)/2} F_{(u + v)/2}, & \text{if $(u - v)/2$ is odd;} \\
	F_{(u - v)/2} L_{(u + v)/2}, & \text{if $(u - v)/2$ is  even,}
	\end{cases}\\
	L_u  + L_v  & = \begin{cases}
	L_{(u - v)/2} L_{(u + v)/2}, & \text{if $(u - v)/2$ is even;} \\
	5F_{(u - v)/2} F_{(u + v)/2}, & \text{if $(u - v)/2$ is odd,}
	\end{cases}\\
	L_u  - L_v  &=  \begin{cases}
	L_{(u - v)/2} L_{(u + v)/2}, & \text{if $(u - v)/2$ is odd;} \\
	5F_{(u - v)/2} F_{(u + v)/2}, & \text{if $(u - v)/2$ is even.}
	\end{cases}
	\end{align*}
\end{proof}

Differentiating~\eqref{main4_k0} $k$ times with respect to $a$ gives the following alternative to \eqref{Main4}.
\begin{theorem}%\label{thm_main4_alt}
	For $m,k\in\mathbb{N}_0$ and $a>0$, we have
	\begin{equation*}
	\begin{split}
	F(m,k,a) &= \int_0^\infty \frac{{x^m\ln x \,dx}}{{(x-1)(x + a)^{k + m + 1}}} \\
	&= \frac{{\pi^2}}{{2(a + 1)^{k + m + 1} }} + \frac{{(- 1)^k }}{2k!\binom{k + m}{m} }\frac{{d^k }}{{da^k }}
	\left( {\frac{{\ln^2 a}}{{(a + 1)^{m + 1} }}} \right) \\
	&\quad\; + \sum_{j = 0}^{m - 1} {\frac{\binom{k + j}j}{\binom{k + m}m}\frac{{H_{m - j - 1} }}{{m - j}}\frac{1}{{(a + 1)^{j + k + 1} }}}
	+ \frac{(- 1)^k}{k!\binom {k + m}{m}}\sum_{j = 0}^{m - 1} \frac{1}{m - j}\frac{{d^k }}{{da^k }}\left( \frac{{\ln  a}}{{(a + 1)^{j + 1} }} \right).
	\end{split}
	\end{equation*}
\end{theorem}
\begin{theorem}
	For $k\in\mathbb{N}_0$ and $a>0$, we have
	\begin{equation}\label{eq.uqky3nr}
	\begin{split}
	\int_0^\infty&\frac{(x^k  - 1)\ln x\,dx}{(x - 1)(x + a)^{k + 1} } \\
	&\qquad= \frac{1}{(a+1)^{k+1}}\sum_{j = 0}^{k - 1} \frac{(1+1/a)^{j+1}}{j+1}\left((a^{j+1}  - 1)H_{ j}  + (a^{j+1}  + 1)\ln a\right),
	\end{split}
	\end{equation}
	\begin{align*}
	\int_0^\infty & \frac{{(x^k  + 1)\ln x\,dx}}{{(x - 1)(x + a)^{k + 1} }}\\
	&\qquad=\frac{{\pi ^2  + \ln ^2 a}}{{(a + 1)^{k + 1} }} + \frac{1}{(a+1)^{k+1}}\sum_{j = 0}^{k - 1} \frac{(1+1/a)^{j+1}}{j+1}\left((a^{j+1}  + 1)H_{ j}  + (a^{j+1}  - 1)\ln a\right) .
	\end{align*}
\end{theorem}
\begin{proof}
	Evaluate $F(k,0,a)\pm F(0,k,a)$.
\end{proof}

Note that~\eqref{eq.uqky3nr} can also be written as
\begin{equation*}
\int_0^\infty  \frac{\ln x \sum\limits_{j=0}^{k - 1}{x^j}}{{(x + a)^{k + 1} }}dx  = \frac{1}{(a+1)^{k+1}}\sum_{j = 0}^{k - 1} \frac{(1+1/a)^{j+1}}{j+1}\left((a^{j+1}  - 1)H_{j}  + (a^{j+1}  + 1)\ln a\right).
\end{equation*}

%\section*{Acknowlegements} Acknowledgements should be put in an unnumbered section.


\begin{thebibliography}{99}
	
	\bibitem{Part30}
	T.~Amdeberhan, A.~Dixit, X.~Guan, L.~Jiu, A.~Kuznetsov, V.~H.~Moll and Ch.~Vignat, {\it The integrals in Gradshteyn and Ryzhik. Part 30: Trigonometric functions,} Sci. Ser. A Math. Sci. (N.S.) 27 (2016), pp. 47--74.
	
	
	\bibitem{BoMo}
	G.~Boros and V.~H. Moll, {\it Irresistible Integrals Symbolics, Analysis and Experiments in the Evaluation of Integrals,} Cambridge University Press, 2004.
	
	\bibitem{GrRy07}
	I.~Gradshteyn and I.~Ryzhik, {\it Table of Integrals, Series, and Products,} Elsevier Academic Press, 2007.
	
	\bibitem{Moll}
	V.~H.~Moll, {\it The integrals in Gradshteyn and Ryzhik. Part 1: A family of logarithmic integrals,}
	Sci. Ser. A Math. Sci. (N.S.), 14 (2007), pp. 1--6.
	
	\bibitem{Moll-Mono1}
	V.~H.~Moll, {\it Special Integrals of Gradshteyn and Ryzhik: The Proofs. Vol. 1,} CRC Press, 2015.
	
	\bibitem{Moll-Mono2}
	V.~H.~Moll, {\it Special Integrals of Gradshteyn and Ryzhik: The Proofs. Vol. 2,} CRC Press, 2016.
	
\end{thebibliography}
\end{document}